\documentclass[12pt]{amsart}

\usepackage{amsmath}
\usepackage{amssymb}
\usepackage{amsthm}
\usepackage{todonotes}
\usepackage{graphicx}
\usepackage{tikz}
\usepackage{xcolor}
\definecolor{lgray}{gray}{.85}

\newtheorem{theorem}{Theorem}
\newtheorem{prop}[theorem]{Proposition}
\newtheorem{lemma}[theorem]{Lemma}
\newtheorem{cor}[theorem]{Corollary}

\allowdisplaybreaks

\begin{document}
\title{Graph Odometry}
\author{Aaron Dutle}
\address{Aaron Dutle \\ Department of Mathematics\\ University of South Carolina \\  Columbia, SC 29208}
\email{dutle@mailbox.sc.edu}
\author{Bill Kay} 
\address{Bill Kay \\ Department of Mathematics \& Computer Science \\  Emory University \\  Atlanta, GA 30329 }
\email{w.w.kay@emory.edu} 

\date{\today}

\begin{abstract}
We address problem of determining edge weights on a graph using non-backtracking closed walks from a vertex. We show that the weights of all of the edges can be determined from any starting vertex exactly when the graph has minimum degree at least three. We also determine the minimum number of walks required to reveal all edge weights.
\end{abstract}
\maketitle
\section{Introduction}
 
 Suppose that a delivery company sets up shop in a small town. Unfortunately, the odometer on its sole  delivery truck is broken. The truck uses exactly one gallon of fuel per mile, regardless of the trip. The delivery truck is also very large, so it cannot turn around on a street or even at an intersection, only at the company lot. At this lot,  it can also refuel, and hence determine the total distance for a trip. The company has a map of the town, and wants to determine the length of each street using information gathered from the delivery truck. What must the map of the town look like in order for this task to be possible? 
 
 This question is naturally framed in the language of edge-weighted graphs, which have been widely studied and are particularly well-suited to modelling real world phenomena. Models of  power grids, computer networks, and telephone networks might have edges that are weighted by the bandwidth, resistance, or the cost to connect nodes. The social network Facebook uses an algorithm to compute an edge weight of a posted item, and shows it to other users based on this score \cite{EdgeRank}. For businesses that use social media for advertising, determining or increasing the \emph{Edgerank} of a posted item on Facebook can be extremely valuable.  
 
 Framed in tha language of edge-weighted graphs, the scenario in the first paragraph asks the following:  If a  graph $G$ has weighted edges, and $v$ is some particular vertex of the graph, can each of the individual edge weights be determined by measuring walks starting and ending at $v,$ where backtracking is not allowed along the walk?    
 
\section{Preliminaries}

Throughout the paper, $G$ is assumed to be a finite, undirected graph. Let $V(G)$ denote the vertex set of $G$ and let $E(G)$ denote the edge set of $G$. Unless otherwise noted, we follow the notation of Diestel \cite{Diestel_2006}.

Define a \emph{walk} in $G$ to be a sequence  $W = \{v_j, v_{j+1}, \ldots , v_k\}\subseteq V(G)$ with $\{v_i, v_{i+1}\} \in E(G)$ for $j \leq i \leq k-1$. We call a walk a
\emph{non-backtracking walk} if we require $v_i \neq v_{i+2}$ for $j \leq i \leq k-2$. Finally, we call a (non-backtracking) walk \emph{closed} if $v_j = v_k$. 

For walks $W_1 = \{v_j, v_{j+1}, \ldots, v_k\}$, $W_2 = \{v_k,v_{k+1}, \ldots, v_\ell\}$ with $j <k<\ell$,   define the binary  operation $\circ$ as follows: $W_1 \circ W_2 := \{v_j,v_{j+1},  \ldots, v_k,v_{k+1}. \ldots, v_\ell\}$, i.e., concatenation of the walks. Define the unary operation $\overline{\cdot}$ as follows: $\overline{W_1} := \{\hat{v}_j = v_k, \hat{v}_2 = v_{k-1}, \ldots , \hat{v}_k = v_j\}$, i.e., reversal of the indices.

Now let $F: E(G) \rightarrow \Omega$ be a weight function of the edge set. We normally take $\Omega$ to be some real interval, although any field or $\mathbb{Z}$-module will do. For ease of notation, let $w_e :=F(e)$ for each $e \in E(G)$. For a closed walk $W= \{v_j, v_{j+1}, \ldots, v_k\}$,
call $F(W) := \sum_{i = j}^{k-1} F(\{v_i, v_{i+1}\})$ the {\em weight} of the walk $W$. Note that $F(\overline{W_1}) = F(W_1)$ and $F(W_1 \circ W_2) = F(W_1) + F(W_2)$.   

The reader can easily verify the following. 
\begin{prop}
Let $W_1$ and $W_2$ be as above. If $W_1$ is a non-backtracking walk, and $W_2$ is a non-backtracking walk, then $\overline{W_1}$ is a non-backtracking walk, and $W_1 \circ W_2$ is a non-backtracking walk
provided that $v_{k-1} \neq v_{k+1} $ \label{concat}
\end{prop}

Let $\mathcal{W}$ be a collection of closed non-backtracking walks in $G$. We say that an edge $e\in E(G)$ is \emph{revealed} by $\mathcal{W}$ if there exist $W_1, W_2, \ldots, W_\ell \in \mathcal{W}$ and non-zero integers $c_e, c_1, c_2, \ldots, c_\ell$ so that  $\sum_{i=1}^\ell c_i F(W_i) = c_ew_e.$ In an analogous way, we say that a walk $W$ is revealed by $\mathcal{W}$ when  there exist $W_1, W_2, \ldots, W_\ell \in \mathcal{W}$ and non-zero integers $c_W, c_1, c_2, \ldots, c_\ell$ so that  $\sum_{i=1}^\ell c_i F(W_i) = c_WF(W).$

For  a vertex $v \in V(G)$, let $\mathcal{W}_v$ denote the set of all closed non-backtracking walks in $G$ starting and ending at $v$. For a subset of vertices $S\subseteq V(G)$, we let $\mathcal{W}_S = \bigcup_{v\in S} \mathcal{W}_v.$ For an edge $e\in E(G)$, we say $e$ is \emph{revealed by  $S$} if $e$ is revealed by $\mathcal{W}_S.$  Define the {\em odometric set} of $S$, denoted $O_S,$ as the set of all edges revealed by $S.$ If $O_S = E(G)$, we say that $G$ is \emph{odometric at $S$}. In the case that $S = \{v\}$ for some vertex $v$, we drop the set notation and say $G$ is odometric at $v.$ Finally, we say that the graph $G$ is \emph{odometric} if it is odometric at $v$ for every vertex  $v\in V(G)$. The main result of this paper is a complete characterization of odometric graphs.

Before moving to the characterization, we note a few subtleties in the above definitions. First, we note that any particular walk  $W\in\mathcal{W}_S$ must have \emph{one} vertex $v\in S$ for the starting and ending vertex. However, we note that there could be walks $W_1, W_2 \in \mathcal{W}_S$ which have different starting and ending vertices. Hence, in revealing some edge, one \emph{is} able to use $W_1$ and $W_2$ despite their starting and ending vertices being different. For example, an edge could be revealed by using three closed walks from $v$, and two closed walks from $w,$ provided $v,w\in S.$

Next, we address the question of using previously revealed edges in revealing later edges. An illustration is perhaps the simplest way to demonstrate. Suppose that $u,v,w$ is a triangle in our graph, and we're trying to discover the odometric set for $v$. Suppose also that we can reveal the edge $e = \{v,u\}$ and the edge $f=\{u,w\}$. One would like to say that we can then reveal $g = \{w,v\}$ by getting the weight of the triangle $W = {v,u,w,v}$ as a closed walk, and then subtracting off the weights of $e$ and $f$ determine the weight of $g$. 

While the definition of revealing an edge says nothing about using the weights of previously revealed edges, our next proposition shows that using such information does not increase the size of an odometric set. 

\begin{prop}\label{prop:direct} 
Let $e$ be an edge of a graph $G,$ and $S\subseteq V(G).$  Suppose that there exist closed walks $W_1, \ldots, W_\ell\in \mathcal{W}_S,$ edges $e_1, \ldots e_m \in \mathcal{O}_{S},$ and nonzero integers $a, b_1, \ldots b_\ell, c_1, \ldots c_m$ such that 
$$aw_e = \sum_{i=1}^\ell b_iF(W_i) + \sum_{j=1}^m c_j w_{e_j}.$$ 
Then $e\in \mathcal{O}_S$. 
\end{prop} 

\begin{proof}
The proof is simple, although notationally cumbersome. By assumption,  for $1\leq j \leq m$, we have $e_j \in \mathcal{O}_S$.  Hence for each of these edges, there exists a collection of closed walks $\{W_1^j,\ldots, W_{\ell_j}^j\} \in \mathcal{W}_S$ and non-zero integers $\alpha_j, \beta_1^j, \beta_2^j, \ldots, \beta_{\ell_j}^j$ so that  $\alpha_j w_{e_j} = \sum_{k=1}^{r_j} \beta_k^j F({W}_k^j).$

Set $\alpha = \prod_{j=1}^m\alpha_j$. Then

\begin{align*} (\alpha a)w_e & = \alpha\left(\sum_{i=1}^\ell b_iF(W_i) + \sum_{j=1}^m c_j w_{e_j}\right) \\
& = \sum_{i=1}^\ell (\alpha b_i)F(W_i) +\sum_{j=1}^m (\frac{\alpha c_j}{\alpha_j}) \alpha_jw_{e_j} \\ 
& = \sum_{i=1}^\ell (\alpha b_i)F(W_i) + \sum_{j=1}^m \sum_{k=1}^{r_j} \left(\frac{\alpha c_j\beta_k^j}{\alpha_j}\right) F({W}_k^j),\end{align*}
which is an expression showing that $e\in \mathcal{O}_S.$ 
\end{proof}

\section{Main Theorem}
There are a few obvious obstructions to a graph being odometric. If $G$ has a vertex of degree $1,$ then it is not odometric. To see this, let $e$ be the leaf edge, and let $v$ be any vertex of the graph other than the given vertex of degree 1. Then $O_v$ cannot contain $e$, as no closed walk from $v$ can even traverse $e$. Similarly, there can be no vertex of degree $2$. In such a graph, if we let $v$ be any vertex other  than the given vertex of degree 2, then the two edges incident to the degree 2 vertex cannot be in $O_v$, since any non-backtracking walk through one of these edges must necessarily traverse the other edge, and so the  weights of the individual edges cannot be separated. We claim that these necessary conditions also suffice. That is, we claim the following theorem:

\begin{theorem}[Main] \label{main}
A (finite) connected, weighted graph $G$ is odometric if and only if the minimum degree of $G$ is 3.
\end{theorem}

The essential structure of the proof is as follows. For any vertex $v$, any non-backtracking walk from $v$ to any cut vertex of a 2-connected block of $G$ can be revealed from $v$ (Lemma \ref{lemma:cutpathmeasuring}). This leads to the fact that the odometric set for $v$ includes the odometric set for any cut vertex of a 2-connected block (Corollary \ref{cor:2cutgas}). Lemma \ref{lemma:cutpathmeasuring} is then extended to reveal any path to any cut vertex (Corollary \ref{lemma:bridgepathmeasuring}),  which gives that the weight of any bridge edge can be revealed (Corollary \ref{cor:bridgeweight}). Finally,  all of the edges of any 2-connected block can be revealed by any vertex in the block (Lemma \ref{lemma:blocks}).

\begin{lemma}\label{lemma:cutpathmeasuring}
Let $G$ be a connected graph with minimum degree 3, let $v\in V(G),$ $B$ be a maximal 2-connected block of $G,$ and $u$ be a cut vertex of $B$. Then any non-backtracking walk $W$ from $v$ to $u$ can be revealed by $v$.

\end{lemma}
\begin{proof} We note that the statement is are vacuous if there are no cut vertices of $B,$ or if $v=u$. 

Let $W = \{ v=v_1, v_2, \ldots, v_k = u\}$. We consider two possibilities. 

\emph{Case 1:} The edge $\{v_{k-1}, v_k\}$ is not contained in $B.$

Because $B$ is 2-connected, $u$ must have two neighbors $x,y$ inside $B$. By Menger's Theorem\footnote{In fact, at both points where we cite Menger's Theorem,  we could use a much weaker statement. We  need only the fact that for any two vertices in a 2-connected graph, there is a cycle containing both of them. The proof is a simple exercise for any reader who desires a self-contained treatment. },
$x$ and $y$ are either adjacent, or there are two internally disjoint paths inside $B$ connecting $x$ to $y$. In either case, there is a path $P = \{x=v_1, \ldots v_\ell = y\}$ (possibly a single edge) from $x$ to $y$ that does not use the vertex $u$. Hence $C = \{u,x\}\circ P \circ \{y,u\}$ is a non-backtracking closed walk from $u$. Also, because $C$ doesn't use the same edge to start and end the closed walk, $C\circ C$ is also a non-backtracking closed walk from $u$ by Proposition \ref{concat}. 

Since $\{v_{k-1}, v_k\}$ is not in $B$, we have that both $W\circ C\circ \overline{W}$ and $W\circ C\circ C\circ \overline{W}$ are non-backtracking closed walks from $v$ by Proposition \ref{concat}. We note that $$ 2F(W) = 2F(W\circ C\circ \overline{W}) - F(W\circ C\circ C\circ \overline{W}),$$ which proves the statement for this case. We also note for use in the next case that $v$ reveals $C$ as well, since $F(C) = F(W\circ C\circ C\circ \overline{W})-F(W\circ C\circ \overline{W})$.

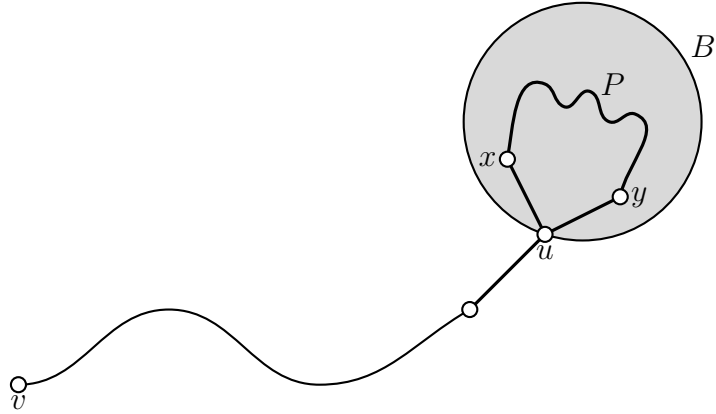
\begin{figure}\label{fig:cutmeasure1}
\begin{tikzpicture}

\draw[thick, fill=lgray] (7.5,4.5) circle [radius = sqrt(1.5^2+.5^2)];
\draw[very thick] (7,3) -- (6,2);
\draw[very thick] (7,3) -- (6.5,4);
\draw[very thick] (7,3) -- (8,3.5);
\draw[very thick](6.5,4) to [in=157.5, out=80] (7,5) to[out=337.5, in=157.5] (7.25,4.7) to[out=337.5, in=157.5] (7.6,4.9) to[out=337.5, in=157.5] (7.85,4.5) to[out=337.5, in=157.5] (8.2,4.6) to [out=337.5, in=80] (8,3.5);
\draw[thick] (6,2) to[out=210, in=0] (4,1) to [out=180,in=0] (2,2) to [out=180, in=0](0,1);
\draw[thick, fill=white] (0,1) circle [radius = .1];
\draw[thick, fill=white] (7,3) circle [radius = .1];
\draw[thick, fill=white] (6,2) circle [radius = .1];
\draw[thick, fill=white] (6.5,4) circle [radius = .1];
\draw[thick, fill=white] (8,3.5) circle [radius = .1];
\node [below] at (0,1) {$v$};
\node [below] at (7,3) {$u$};
\node [left] at (6.5,4){$x$};
\node [right] at(8,3.5){$y$};
\node at (9.1, 5.5) {$B$};
\node at (7.9,5){$P$};
\end{tikzpicture}
\caption{The first case of Lemma \ref{lemma:cutpathmeasuring}}
\end{figure}

\emph{Case 2:} The edge $\{v_{k-1}, v_k\}$ is contained in $B.$

Consider the structure of $G$. Recall that in the block graph $G$, there are vertices for each block of $G$ (isolated vertex, bridge, or maximal 2-connected component), and vertices for each cut vertex. Two vertices in the block graph are adjacent when one corresponds to a block, the other to a cut vertex, and the cut vertex is contained in the block. Recall also that the block graph of a connected graph is always a tree, whose leaf vertices are bridges or 2-connected blocks \cite{Diestel_2006}.

In our case, the leaf vertices in the block graph cannot be bridges, as these would be vertices with degree 1, contradicting the hypotheses. Hence all leaves correspond to 2-connected components. In the block graph, if we remove the edge connecting $B$ to $u,$ the component containing $u$ is still a (nontrivial) tree. Hence there is a path in this tree from $u$ to a leaf, which corresponds to a different 2-connected block $B'.$ The neighbor of this leaf corresponds to a cut vertex $u',$ (where it is possible that $u=u'$.) Thus in the block graph, we have a path from the vertex corresponding to $u$ to a vertex corresponding to a cut vertex $u'$ of a 2-connected block $B' \neq B.$ Also, this path avoids our original block $B.$    

For each block vertex traversed by this path, we can find a non-backtracking path through the block in our original graph $G.$ Hence in $G,$ there is a (possibly trivial) non-backtracking path $W'$ from $u$ to a cut vertex $u'$ of $B'.$  As the path in the block graph avoided $B$, we see that $W'$ has no edges inside $B.$ 

Note that $W\circ W'$ is then a non-backtracking walk from $v$ to a cut-vertex $u'$ of a maximal 2-connected component $B'$ of $G$, and the last edge of this walk is not inside the block. The proof of Case 1 provides a cycle $C'$ in $B'$ containing $u',$ and reveals $W\circ W'$  and  $C'$. 

Next, we note that $W\circ W' \circ C' \circ \overline{W'}$ satisfies the conditions of Case 1 for our original cut vertex $u$ and block $B.$ Hence we can determine $F(W\circ W' \circ C' \circ \overline{W'}).$ From this, we can deduce $F(W')$ by  
$$F(W') = F(W\circ W' \circ C' \circ \overline{W'}) - F(W\circ W')-F(C').$$

Finally, we see that $F(W) = F(W\circ W')- F(W'),$ proving the statement in the second case. 
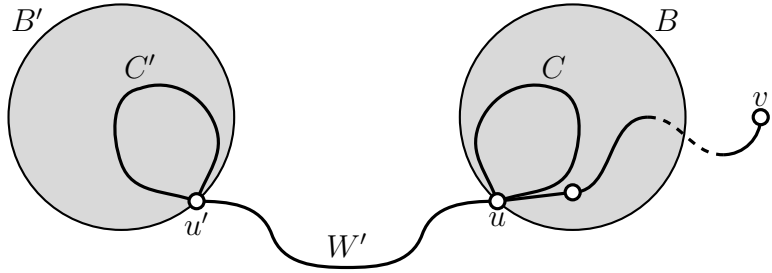
\begin{figure}\label{fig:cutmeasure2}
\begin{tikzpicture}

\draw[very thick] (0,0) to [out=0, in=250] (1,.44) to[out=70, in=180] (2,.882);
\draw[thick, fill=lgray] (3,2) circle [radius=1.5];
\draw[very thick] (2,.882) to[out=110, in=250] (1.75,1.882) to[out=70, in=160](2.75,.882+1.5) to[out=350, in=70] (3,.882+.5) to[out=250, in=20] (2,.882);
\draw[ very thick] (3,1) -- (2,.882);
\draw[very thick] (3,1) to [out=0, in=180] (4,2);
\draw[ dashed, very thick] (4,2) to [out=0, in=180](5,1.5);
\draw[ very thick] (5,1.5) to[ out=0, in=270] (5.5,2); 
\draw[very thick, fill=white] (5.5,2) circle [radius=.1];
\draw[very thick, fill=white] (3,1) circle [radius=.1];
\draw[very thick, fill=white]  (2,.882) circle [radius=.1];
\node [below] at (2,.882) {$u$};
\node [above] at (2.75,.882+1.5) {$C$};
\node [above] at (5.5,2) {$v$};
\node [above] at (4.25, 3) {$B$};

\draw[very thick] (0,0) to [out=180, in=180-250] (-1,.44) to[out=180-70, in=0] (-2,.882);
\draw[thick, fill=lgray] (-3,2) circle [radius=1.5];
\draw[very thick] (-2,.882) to[out=70, in=180-250] (-1.75,1.882) to[out=180-70, in=180-160](-2.75,.882+1.5) to[out=180-350, in=180-70] (-3,.882+.5) to[out=180-250, in=180-20] (-2,.882);
\draw[very thick, fill=white]  (-2,.882) circle [radius=.1];
\node [below] at (-2,.882) {$u'$};
\node [above] at (-2.75,.882+1.5) {$C'$};
\node [above] at (-4.25, 3) {$B'$};
\node [above] at (0,0) {$W'$};
\end{tikzpicture}
\caption{The second case of Lemma \ref{lemma:cutpathmeasuring}}
\end{figure}
\end{proof} 

We can use this Lemma and its proof technique to deduce an important corollary. 

\begin{cor} \label{cor:2cutgas}
Let $G$ be a connected graph with minimum degree 3, $v\in V(G)$, and $u$ be a cut vertex for a $2$ connected block of $G$. Then $O_u \subseteq O_{v}$.  
\end{cor}

\begin{proof} 
Fix a cut vertex $u,$ and a 2 connected block $B$ that contains it. It suffices to show that any closed walk from $u$ can be simulated by a closed walk from $v,$ in the sense that the weight of the walk can be determined using closed walks from $v$.  Let $W = \{u=v_1, v_2, \ldots, v_k=u\}$ be such a closed walk. 

First, we claim that, without loss of generality, this walk has its starting and ending edges either both in $B$, or neither in $B.$ To see this, consider the components of $G\setminus \{u\}$, and suppose that the first edge of the walk is some fixed component.  As $u$ is a cut vertex, the walk must pass through $u$ from that same component to reach a different one. Hence we can write any closed walk from $u$ as the concatenation of closed walks, each of which begin and end in a single component of $G\setminus \{u\}.$ To determine the weight of $W$, we can determine the weights of each closed subwalk inside a component, and add the weights. Hence we can assume the starting and ending edges of $W$ are either both in $B,$ or neither in $B.$ 

Next, fix a path $P = \{v=v_1, \ldots, v_k = u\}$ from $v$ to $u.$ We define a non-backtracking walk $Q$ based on the last edge of $P.$ If $\{v_{k-1}, v_k\} \notin B,$ then we use the cycle $C$ given by the proof of Case 1 of the previous Lemma, and set $Q = P\circ C.$ This makes $Q$ a non-backtracking walk from $v$ to $u$ with last edge \emph{in} $B$. If $\{v_{k-1}, v_k\} \in B,$ we use the walk $W'$ and cycle $C'$ given by the proof of Case 2 of the Lemma, and set $Q = P \circ W' \circ C' \circ \overline{W'}.$ This makes $Q$ a non-backtracking walk from $v$ to $u$ with last edge \emph{not} in $B.$ 

Hence $\{P, Q\}$ are two fixed walks, one with last edge in $B$, and the other with last edge not in $B.$ Without loss of generality, assume $P$ has last edge outside $B.$ Then for any closed walk $W$ from $u$ with first and last edge in $B,$ we see that $P\circ W \circ \overline{P}$ is a closed walk from $v$. We can reveal $P$ by the previous Lemma, so we can also reveal $W.$  If $W$ begins and ends outside $B$, we can use the walk $Q\circ W \circ \overline{Q}$ in the same manner to reveal $W$. 

Hence the weight of any closed walk from $u$ can be found using closed walks from $v$, and the proof is complete.

\end{proof}          
With only slightly more effort, we can obtain the same result of Lemma \ref{lemma:cutpathmeasuring}  for \emph{any} cut vertex. Using this, we'll be able to show that any bridge can be revealed by any vertex.     
\begin{lemma}
\label{lemma:bridgepathmeasuring}
Let $G$ be a connected graph with minimum degree 3,  $v\in V(G),$ and $u$ be a cut vertex of $G$. Then  any non-backtracking walk $W$ from $v$ to $u$ can be revealed by $v$.

\end{lemma}

\begin{proof}
If $u$ is in any 2-connected component, Lemma \ref{lemma:cutpathmeasuring} applies, and we have the result. So we can assume that every edge incident to $u$ is a bridge. 

Let $W = \{v=v_1, v_2, \ldots, v_k = u\} $ be our non-backtracking walk. As our graph has minimum degree 3, $u$ must have two neighbors $x$ and $y$ that are distinct from $v_{k-1}$, the second to last vertex in our walk $W$. Let $e = \{u,x\}$ and $f=\{u,y\}$. In the block graph of $G$, removing the vertex corresponding to $u$ splits the block graph into a collection of trees. Since $u$ is a cut vertex, there are two distinct, nontrivial trees $T_e$ and $T_f$ containing the vertices corresponding to $e$ and $f$ respectively. In $T_e$, we can find a path from the vertex corresponding to $e$ to a leaf, which itself corresponds to a 2-connected component. As in the proof of Lemma \ref{lemma:cutpathmeasuring}, we can use this to find a closed walk $C_e$ from $x$ that never crosses $e$. Similarly, we can find $C_f$ from $y.$ Now we define $C = e\circ C_e \circ e \circ f \circ C_f \circ f$ and observe that this is a closed walk from $u$ that starts and ends on two distinct edges, both of which are different from the last edge of $W$. Hence $W\circ C\circ \overline{W}$ and $W\circ C\circ C \circ \overline{W}$ are both closed walks from $v$. Noting that
$$2F(W) = 2F(W\circ C\circ \overline{W}) - F(W\circ C\circ C\circ\overline{W}),$$
our proof is complete.  
\end{proof} 

\begin{cor} \label{cor:bridgeweight}
Let $G$ be a connected graph with minimum degree 3, and $v\in V(G)$. If  $e = \{u, u'\}$ is a bridge in $G$, then $e\in \mathcal{O}_v$. 
\end{cor}

\begin{proof} Let $e$ be a bridge. We first note that $G$ must contain some 2-connected component, as its block graph is a non trivial tree, and all of the leaves in this tree correspond to 2-connected components. Let $n$ be the number of edges in the shortest path connecting $e$ to a 2-connected component. 

If $n=0$, then one of the vertices of $e$ (without loss of generality, $u$) is a cut vertex of a 2-connected component. By Corollary \ref{cor:2cutgas}, $\mathcal{O}_u\subseteq \mathcal{O}_v$. 
By Lemma \ref{lemma:bridgepathmeasuring}, applied to the walk $\{u, u'\}$ from $u,$ we have that $e\in \mathcal{O}_u$, and hence $e\in \mathcal{O}_v$. 

If $n>0$, let the shortest path from $e$ to a 2-connected component be $P = \{u=v_0, v_1, \ldots, v_n\}.$ Note that $u$ is itself a cut vertex, that $v_1\neq u'$, and that $v_n$ is a cut vertex of a 2-connected component. Also, both $\overline{P}$ and $\overline{P}\circ e$ are non-backtracking walks from $v_n$.   Lemma \ref{lemma:bridgepathmeasuring} says that we can reveal both  $\overline{P}$ and $\overline{P}\circ e$ from vertex $v_n.$ As $w_e = F(\overline{P}\circ e) - F(\overline{P}),$    we see that $e\in \mathcal{O}_{v_n}.$ Corollary \ref{cor:2cutgas} says that $\mathcal{O}_u\subseteq \mathcal{O}_v$, completing the proof.    
\end{proof}
 
 Next we prove a general statement about revealing edges, which will prove useful in the sequel.  

\begin{lemma}\label{lemma:neighbor}
Let $v \in V(G)$, and let $u$ be a neighbor of $v$. If $f = \{u,v\}$ is revealed by $v$, then $O_{\{u ,v\}} = O_v$. 
\end{lemma}
\begin{proof}
One containment is clear: $O_v \subseteq O_{\{u,v\}}$ by definition. Now suppose $e \in O_{\{u,v\}}$. Then there are constants $c_e, c_1, c_2 , \ldots , c_\ell$ and walks $W_1, W_2, \ldots, W_\ell \subseteq W_{\{u,v\}}$ so that $\sum_{i = 1}^\ell c_iF(W_i) = c_e w_e$. We will build a set $\{W_1^\prime, W_2^\prime, \ldots, W_\ell^\prime\} \subseteq \mathcal{W}_v$ that also reveals $e$. 
Note that for $1 \leq i \leq \ell$,  either $W_i$
$\in \mathcal{W}_v$ or $W_i \in \mathcal{W}_u$ by construction. If the walk is in $\mathcal{W}_v$,  set $W_i^\prime = W_i$  Otherwise,  build a new walk $W_i^\prime$. 

Fix such a closed walk from $u$, $W_i = \{u=v_1, v_2, \ldots, v_k=u\}$. We build our walk based on the vertices $v_2$ and $v_{k-1}.$ 

\emph{Case 1:}  $v\notin \{v_2, v_{k-1}\}$.  

 $W_i^\prime = \{v,u\}\circ W_i \circ \{u,v\}$ is a closed non-backtracking walk from $v$ by Proposition \ref{concat}. 
 
\begin{figure}\label{fig:neighbor1}
\begin{tikzpicture}
\draw[thick] (0,2) -- (1,2);
\draw[very thick] (1,2) to [out=0, in=270](2,2.5)to [out=90, in= 90](3,3) to[out=270, in=90](2.5,2) to [out=270, in=90](3,1) to [out=270, in=270] (2,1) to [out=90, in=270] (1,2);
\draw[thick, fill=white] (0,2) circle [radius=.1];
\draw[thick, fill=white] (1,2) circle [radius=.1];
\node [above] at (0,2) {$v$};
\node [above] at (1,2) {$u$};
\node [below] at (.5,2) {$e$};
\node [right] at (2.5,2) {$W_i$};

\end{tikzpicture}
\caption{The first case of Lemma \ref{lemma:neighbor}}
\end{figure}
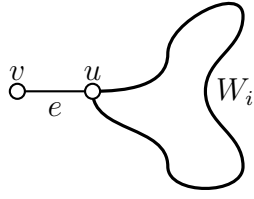
\emph{Case 2:} Exactly one of the vertices in $\{v_2, v_{k-1}\}$ is the vertex $v$.

 Note that by exchanging $W_i$ for $\overline{W}_i$, we may assume $v_2 = v$. Then $W_i = \{ u = v_1, v = v_2, v_3, \ldots, v_{k-1}, v_k = u\}$. In this case, we simply shift the starting point of the walk to $v$. Specifically, we let $W_i^\prime = \{v=v_2, v_3, \ldots, v_{k-1}, v_k = u, v_{k+1}= v\}$. It is clear that the first portion $\{v_2, v_3, \ldots, v_k\}$ constitutes a non-backtracking walk since it is a subwalk of $W_i$. Moreover, since $v_{k-1} \neq v$, and $\{v_k, v_{k+1}\} \in E(G)$, so we can append $v = v_{k+1},$ making $W_i^\prime$ a non-backtracking closed walk. 

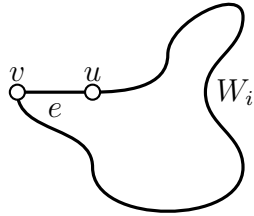
\begin{figure}\label{fig:neighbor2}
\begin{tikzpicture}
\draw[very thick] (0,2) -- (1,2);
\draw[very thick] (1,2) to [out=0, in=270](2,2.5)to [out=90, in= 90](3,3) to[out=270, in=90](2.5,2) to [out=270, in=90](3,1) to [out=270, in=270] (1,1) to [out=90, in=270] (0,2);
\draw[thick, fill=white] (0,2) circle [radius=.1];
\draw[thick, fill=white] (1,2) circle [radius=.1];
\node [above] at (0,2) {$v$};
\node [above] at (1,2) {$u$};
\node [below] at (.5,2) {$e$};
\node [right] at (2.5,2) {$W_i$};

\end{tikzpicture}
\caption{The second case of Lemma \ref{lemma:neighbor}}
\end{figure}
\emph{Case 3:}  $v_2 = v_{k-1} = v$. 

Let $W_i^\prime = \{v_2, \ldots, v_{k-1}\}$ Note that by assumption this walk is closed at $v$, and it is a non-backtracking walk since it is a sub-walk of $W_i$. 

\begin{figure}\label{fig:neighbor3}
\begin{tikzpicture}
\draw[very thick] (0,2) -- (-1,2);
\draw[very thick] (-1,2) to [out=0, in=270](-2,2.5)to [out=90, in= 90](-3,3) to[out=270, in=90](-2.5,2) to [out=270, in=90](-3,1) to [out=270, in=270] (-2,1) to [out=90, in=270] (-1,2);
\draw[thick, fill=white] (0,2) circle [radius=.1];
\draw[thick, fill=white] (-1,2) circle [radius=.1];
\node [above] at (0,2) {$u$};
\node [above] at (-1,2) {$v$};
\node [below] at (-.5,2) {$e$};
\node [left] at (-2.5,2) {$W_i$};
\end{tikzpicture}
\caption{The third case of Lemma \ref{lemma:neighbor}}
\end{figure}
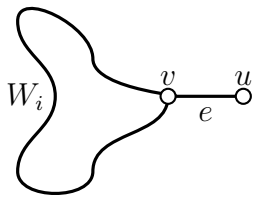

Now, to show that $e$ is revealed by $v$, we first  note that in obtaining $W_i^\prime$ from $W_i$, the only edge ever added or deleted was $f$. Hence we have that $F(W_i^\prime) = F(W_i)+\epsilon_iw_f,$ where the value of $\epsilon_i$ depends on how we obtained $W_i^\prime$ from $W_i$ (in fact, $\epsilon_i\in \{-2,0,2\}$). 

If we reuse the constants $c_e, c_1, \ldots, c_\ell$ used in revealing $e$ as above, we find that
\begin{align*} c_e w_e & = \sum_{i=1}^\ell c_i F(W_i)\\
& =\sum_{i=1}^\ell c_i (F(W_i)+\epsilon_iw_f) -\left(\sum_{i=1}^\ell c_i\epsilon_i\right)w_f \\
&= \sum_{i=1}^\ell c_i F(W_i^\prime) + \epsilon w_f. \end{align*} 

Here $\epsilon = -\sum_{i=1}^\ell c_i\epsilon_i$.  Since we're assuming that $f \in \mathcal{O}_v,$  Proposition \ref{prop:direct} tells us that $e\in \mathcal{O}_v$, completing the proof.     
\end{proof}

\begin{lemma} \label{lemma:blocks}
Let $B$ be a maximal 2-connected block of a graph $G$ with minimum degree 3, and let $v$ be any vertex in $V(B)$. Then $E(B) \subseteq \mathcal{O}_v$. 
\end{lemma}

\begin{proof} Let $e = \{u,v\}$. We prove that $e\in \mathcal{O}_v$ using two cases.   

Suppose  $u$ is a cut vertex of $G$. Then Lemma \ref{lemma:cutpathmeasuring} applies to $e,$ and so $v$ reveals $e.$ 

Suppose that  $u$ is \emph{not} a cut vertex of $G.$ Because $G$ has minimum degree 3, $u$ must have two neighbors $x$ and $y$ that are distinct from $v.$ As $u$ is not a cut vertex, both $x$ and $y$ are contained in $B.$  By Menger's Theorem, $x$ and $y$ are either adjacent, or there are two internally disjoint paths inside $B$ connecting $x$ to $y$. In either case, there is a path $P$ from $x$ to $y$ avoiding $u.$ Then $C = \{u,x\}\circ P\circ \{y,u\}$ is a closed non-backtracking walk from $u$. As the start and end edges are different, $C\circ C$ is a closed non-backtracking walk from $u$ also. Because neither the start nor end edges are $e$, we can append $e$ to the start and end of each of these walks to obtain closed non-backtracking walks from $v$. Noting that 

$$2w_e = 2F(e\circ C \circ e)- F(e\circ C\circ C\circ e),$$
we see that $v$ reveals the edge $e$.         

Since $e$ was an arbitrary edge incident to $v$, we have that $v$ reveals all of its incident edges. Now we employ Lemma \ref{lemma:neighbor}, which says that for each neighbor $u \in N(v)$, that $\mathcal{O}_u\subseteq \mathcal{O}_v$. Then our proof above applies to the neighbors of $v$, showing that all the edges within the second neighborhood of $v$ are also in $\mathcal{O}_v.$ Continuing the process as many times as needed (the diameter of $B$ times would suffice), we see that $E(B) \subseteq \mathcal{O}_v.$ 
\end{proof} 

Now we are prepared to prove  the main theorem, that any connected graph with minimum degree 3 is odometric.  

\begin{proof} [Proof of Theorem \ref{main}] 
Let $G$ be a graph with minimum degree 3, let $v\in V(G),$ and let $e$ be an edge of $G$. If $e$ is a bridge, Corollary \ref{cor:bridgeweight} says that $e\in\mathcal{O}_v$. If $e$ is not a bridge, it must be in some 2-connected block $B\subseteq G$. If $v\in B$, then Lemma \ref{lemma:blocks} gives that $e\in \mathcal{O}_v$. If $v\notin B$, then since $G$ is connected, there must be a cut vertex $u \in B.$ By  Lemma \ref{lemma:blocks}, we have  $e\in \mathcal{O}_u,$ and by Corollary \ref{cor:2cutgas}, we have $\mathcal{O}_u\subseteq \mathcal{O}_v$.  

As $e$ was arbitrary, we see that $E(G) = \mathcal{O}_v$. As $v$ was arbitrary, we see that $G$ is odometric.

\end{proof}

\section{Minimal Sets of Walks}

If $G$ is odometric, and $v\in V(G)$, the proof of Theorem \ref{main} can easily be turned into an algorithm for \emph{constructing} a collection of closed non-backtracking walks $\mathcal{W}$ from vertex $v$ which reveal every edge of $G$.  Note that the collection of \emph{all} such walks, which we denoted $\mathcal{W}_v,$ is infinite.\footnote{Once a cycle has been found in some walk, it can be repeated  any number times desired, each of which is a different walk. } Hence   $\mathcal{W} \subsetneq \mathcal{W}_v$. The question then arises as to finding a \emph{minimal} (in some sense) set of walks from $v$ that reveal the edges of $G$. If we take minimal to mean the cardinality of $\mathcal{W}$, the answer is a simple exercise in linear algebra.

\begin{theorem}\label{thm:minimal}
If $G$ is odometric, and $v\in V(G)$, then for any minimal collection $\mathcal{W}$ of non-backtracking closed walks  from $v$ that reveals every edge of $G,$ we have $|\mathcal{W}| = |E(G)|.$
\end{theorem}

\begin{proof}
Create a $|E(G)|\times |\mathcal{W}|$ integer matrix $M$ by setting $M(e,W)$ to be the number of times the walk $W$ uses the edge $e$. 

We claim first that the columns of $M$ must be linearly independent over $\mathbb{Q}$. To see this, denote the columns of $M$ by $\{\mathbf{w}_i\}_{i\leq |\mathcal{W}|}$ and suppose that 
$$\mathbf{w} = \sum_{i=1}^ka_k\mathbf{w}_i.$$ If we let $\mathbf{E}$ be the row vector of the edge weights, we see that 
\begin{align*} F(W) & = \langle \mathbf{w}, \mathbf{E}\rangle \\
& =  \sum_{i=1}^ka_k\langle\mathbf{w}_i, \mathbf{E}\rangle \\ 
& = \sum_{i=1}^ka_kF(W_i).
\end{align*} 
Hence any use of $F(W)$ in an integer expression to reveal an edge can be replaced by  $\sum_{i=1}^ka_kF(W_i),$ which (after clearing the denominators of the $a_k$) will be another expression revealing the same edge, without using $F(W)$. This contradicts of the minimality of $\mathcal{W},$ proving that the columns of $M$ are independent. Since the columns are independent, there cannot be more of them than the number of rows. Hence $|\mathcal{W}|\leq |E(G)|.$ 

Next we show that $|\mathcal{W}|\geq |E(G)|$. Create a $|\mathcal{W}|\times|E(G)|$ matrix $N$ by setting $N(W,e)$ to be the coefficient of $F(W)$ in the integer expression that reveals edge $e$. One can easily check that $MN$ is a $|E(G)|\times|E(G)|$ diagonal matrix where $MN(e,e)$ is the (nonzero) coefficient of $F(e)$ in the expression that reveals edge $e$.  Hence the matrix $M$ has full row rank, and so $|\mathcal{W}|\geq |E(G)|.$

\end{proof} 

Theorem \ref{thm:minimal} says that a minimal (cardinality) set of walks $\mathcal{W}$ that reveals the graph contains the same number of walks as there are edges in the graph. On the other hand, revealing any particular edge requires at least two such walks, so there is no obvious bijection between the two sets! 
\section{Conclusion}
The main theorem (Theorem \ref{main}) completely classifies odometric graphs, but there are still many questions that can be asked about graph odometry itself, and its relation to other areas.

  It would be interesting to find applications for an \emph{odometric representation} of a graph $G,$ where an odometric representation is some presentation of a  collection $\mathcal{W}$ that reveals the graph. Can such a representation be used to uncover other information about the structure of the graph? Theorem \ref{thm:minimal} shows that every minimal set of walks revealing the graph has the same size. On the other hand, the actual collection of walks depends crucially on the vertex $v$ chosen as the starting point. Perhaps an odometric representation contains information about what the graph ``looks like'' when viewed from the perspective of vertex $v$. 
  
  Another possible avenue to explore would be to look for odometric sets of walks from $v$ that satisfy some other conditions for minimality. For example, minimize the sum of the weights of the walks in $\mathcal{W}$. Using this notion, one could then minimize over the choice of starting vertices as well, and find an \emph{odometrically minimal} vertex of the graph $G$. In the context of odometry, such a vertex could be considered as the center of the graph. Is such a central vertex unique?  
  
  It should be noted that in defining a non-backtracking \emph{closed} walk, the starting and ending point plays a special role, in that the first and last edges can be the same. If the walk had no particular start, this step would normally be considered a backtracking step. Call such a walk a \emph{strongly} non-backtracking closed walk from $v$.  Call a graph \emph{strongly odometric} if it is odometric at every vertex $v$, but where we restrict the set $\mathcal{W}_v$ to  be the strongly non-backtracking closed walks at $v$. In such a situation, almost all of the techniques used in this paper no longer work. A characterization for such graphs seems to be a much more difficult problem. 

It would also be interesting to explore the possible relationship of graph odometry to statistical mechanics. In 1960, Sherman \cite{Sherman_1960} showed that there is a connection between the Ising model on a planar graph and the collection of non-backtracking walks on the graph. In particular, the generating function for even subgraphs of a finite planar graph $G$ can be expressed as the exponential of a sum of weighted non-backtracking walks in $G$. 

\section{Acknowledgements}
The authors would like to thank Ser-Wei Fu, who first presented the question, and noticed the obvious necessary conditions. We also want to thank Joshua Cooper, who read and edited early drafts of the paper.


\begin{thebibliography}{99}
\bibitem{Diestel_2006} R.~Diestel, \emph{Graph Theory}, Springer, 3rd edition, (2005).

\bibitem{EdgeRank} J.~Kincaid, Edgerank: The secret sauce that makes
Facebook's news feed tick. \verb|http://techcrunch.com/2010/04/22/facebook-edgerank/|, April 2010.

 
\bibitem{Sherman_1960} S.~Sherman, Combinatorial Aspects of the Ising Model for Ferromagnetism I: A Conjecture of Feynman on Paths and Graphs, \emph{Journal of Mathematical Physics}, 1(3):202, 1960.
\end{thebibliography}
\end{document}